\documentclass[12pt]{article}
\usepackage{amsthm,amsfonts,amssymb,amsmath}
\usepackage{cite,hyperref}
\usepackage{url,xcolor}
\usepackage{authblk,fullpage}
\usepackage[shortlabels]{enumitem}
\numberwithin{equation}{section}

\newcommand{\N}{\mathbb{N}}

\newcommand{\R}{\mathbb{R}}

\newtheorem{thm}[equation]{Theorem}
\newtheorem{cor}[equation]{Corollary}
\newtheorem{lem}[equation]{Lemma}
\newtheorem{prop}[equation]{Proposition}

\newtheorem{ques}[equation]{Question}

\theoremstyle{definition}

\theoremstyle{remark}

\newcommand{\newConj}[2]{
  \expandafter\newcommand\csname conjText#1\endcsname{#2}
\begin{conj}
\label{conj:#1}
\csname conjText#1\endcsname
\end{conj}
}

\newcommand{\newQues}[2]{
  \expandafter\newcommand\csname quesText#1\endcsname{#2}
\begin{ques}
\label{ques:#1}
\csname quesText#1\endcsname
\end{ques}
}

\newcommand{\newProb}[2]{
  \expandafter\newcommand\csname probText#1\endcsname{#2}
\begin{prob}
\label{prob:#1}
\csname probText#1\endcsname`
\end{prob}
}

\newcommand{\newThm}[2]{
  \expandafter\newcommand\csname thmText#1\endcsname{#2}
\begin{thm}
\label{thm:#1}
\csname thmText#1\endcsname
\end{thm}
}

\newcommand{\newLem}[2]{
  \expandafter\newcommand\csname lemText#1\endcsname{#2}
\begin{lem}
\label{lem:#1}
\csname lemText#1\endcsname
\end{lem}
}

\newcommand{\repeatConj}[1]{
\csname theoremstyle\endcsname{plain}
\csname newtheorem\endcsname*{#1ConjRepeat}{Conjecture~\csname ref\endcsname{conj:#1}}
\csname begin\endcsname{#1ConjRepeat}
\csname conjText#1\endcsname
\csname end\endcsname{#1ConjRepeat}
}

\newcommand{\repeatQues}[1]{
\csname theoremstyle\endcsname{definition}
\csname newtheorem\endcsname*{#1QuesRepeat}{Question~\csname ref\endcsname{ques:#1}}
\csname begin\endcsname{#1QuesRepeat}
\csname quesText#1\endcsname
\csname end\endcsname{#1QuesRepeat}
}

\newcommand{\repeatProb}[1]{
\csname theoremstyle\endcsname{definition}
\csname newtheorem\endcsname*{#1ProbRepeat}{Problem~\csname ref\endcsname{prob:#1}}
\csname begin\endcsname{#1ProbRepeat}
\csname probText#1\endcsname
\csname end\endcsname{#1ProbRepeat}
}

\newcommand{\repeatThm}[1]{
\csname theoremstyle\endcsname{plain}
\csname newtheorem\endcsname*{#1ThmRepeat}{Theorem~\csname ref\endcsname{thm:#1}}
\csname begin\endcsname{#1ThmRepeat}
\csname thmText#1\endcsname
\csname end\endcsname{#1ThmRepeat}
}

\newcommand{\repeatLem}[1]{
\csname theoremstyle\endcsname{plain}
\csname newtheorem\endcsname*{#1LemRepeat}{Lemma~\csname ref\endcsname{lem:#1}}
\csname begin\endcsname{#1LemRepeat}
\csname lemText#1\endcsname
\csname end\endcsname{#1LemRepeat}
}

\title{An Approximate Counting Version of the Multidimensional Szemer\'edi Theorem}
\author{Natalie Behague\thanks{Research supported by a PIMS Postdoctoral Fellowship.}}
\author{Joseph Hyde}
\author{Natasha Morrison\thanks{Research supported by NSERC Discovery Grant RGPIN-2021-02511 and NSERC Early Career Supplement DGECR-2021-00047 and a Start-Up Grant from the University of Victoria.}}
\author{Jonathan A. Noel\thanks{Research supported by NSERC Discovery Grant RGPIN-2021-02460 and NSERC Early Career Supplement DGECR-2021-00024 and a Start-Up Grant from the University of Victoria.}} 
\author{Ashna Wright\thanks{Research supported by an NSERC CGS-M.}} 
\affil{\normalsize{Department of Mathematics and Statistics, University of Victoria, Victoria, B.C., Canada.}}
\affil{\texttt{\{nbehague,josephhyde,nmorrison,noelj,ashnawright\}@uvic.ca}}

\renewcommand{\vec}{\boldsymbol}

\begin{document}

\maketitle

\begin{abstract}
For any fixed $d\geq1$ and subset $X$ of $\mathbb{N}^d$, let $r_X(n)$ be the maximum cardinality of a subset $A$ of $\{1,\dots,n\}^d$ which does not contain a subset of the form $\vec{b} + rX$ for $r>0$ and $\vec{b} \in \mathbb{R}^d$. Such a set $A$ is said to be \emph{$X$-free}. The Multidimensional Szemer\'edi Theorem of Furstenberg and Katznelson states that $r_X(n)=o(n^d)$. We show that, for $|X|\geq 3$ and infinitely many $n\in\mathbb{N}$, the number of $X$-free subsets of $\{1,\dots,n\}^d$ is at most $2^{O(r_X(n))}$. The proof involves using a known multidimensional extension of Behrend's construction to obtain a supersaturation theorem for copies of $X$ in dense subsets of $[n]^d$ for infinitely many values of $n$ and then applying the powerful hypergraph container lemma. 
%The main ingredients in our proof are a known extension of the Behrend Construction to higher dimensions, a supersaturation theorem for $X$-free sets, and the powerful hypergraph container lemma. 
Our result generalizes work of Balogh, Liu, and Sharifzadeh on $k$-AP-free sets and Kim on corner-free sets. 
\end{abstract}

\section{Introduction}

Szemer\'edi's Theorem~\cite{szemeredi} states that the maximum cardinality of a subset of $[n]:=\{1,\dots,n\}$ not containing $k$ points in an arithmetic progression is $o(n)$, settling a conjecture of Erd\H{o}s and Tur\'{a}n~\cite{erdos-conjecture} from the 1930s. Szemer\'edi's original proof was purely combinatorial and led to the discovery of his celebrated Regularity Lemma~\cite{reg-lemma}. A few years later, Furstenberg~\cite{furstenberg} found an alternative proof using ergodic theory. While this approach relies on the Axiom of Choice and, therefore, does not yield effective bounds, it has the benefit of being more amenable to generalization. For example, the first proof of the Density Hales--Jewett Theorem was obtained via an extension of Furstenberg's ideas~\cite{density-furstenberg}. Later, another proof of Szemer\'edi's Theorem was given by Gowers~\cite{gowers-fourier} using Fourier analysis. 

Our main result is a ``counting version'' of the Multidimensional Szemer\'edi Theorem (Theorem~\ref{thm: multi-sz} below), which was first proven by Furstenberg and Katznelson~\cite{katznelson} using ergodic theory. Later, using extensions of Szemer\'edi's Regularity Lemma to hypergraphs, Gowers~\cite{gowers-multi}, Tao~\cite{tao} and R\"odl, Nagle, Schacht, and Skokan~\cite{skokan, nagle} obtained combinatorial proofs of the Multidimensional Szemer\'edi Theorem which yield effective bounds (albeit, with very extreme dependencies). In order to state our result, we require a few definitions. Given $d\in \mathbb{N}$ and a set $X\subseteq \mathbb{N}^d$, a \emph{copy} of $X$ is a set of the form
\[\vec{b}+rX :=\{\vec{b}+r\vec{x}: \vec{x}\in X\},\]
where $\vec{b}\in \mathbb{R}^d$ and $r\in \mathbb{R}_{\geq 0}$. A copy of $X$ is \emph{non-trivial} if $r\neq 0$. Let $r_X(n)$ denote the cardinality of the largest subset of $[n]^d$ which does not contain a non-trivial copy of $X$; such a set is said to be \emph{$X$-free}. For $d=1$, we simply write $r_k(n)$ to mean $r_{\{1,\dots,k\}}(n)$. In this language, Szemer\'edi's Theorem~\cite{szemeredi} says that $r_k(n)=o(n)$, for any $k\geq3$. The Multidimensional Szemer\'edi Theorem of~\cite{katznelson} is as follows.

\begin{thm}[Multidimensional Szemer\'edi Theorem~\cite{katznelson}; see also~\cite{gowers-multi,tao,skokan,nagle}]\label{thm: multi-sz}
If $d\geq1$ and $X$ is a finite subset of $\N^d$, then $r_X(n) = o(n^d)$. 
\end{thm}

Our focus in this paper is on the related question of counting the number of $X$-free subsets of $[n]^d$. It is trivial to see that there are at least $2^{r_X(n)}$ such sets. The following well-known (and still open) question of Cameron and Erd\H{o}s~\cite{erdos-question} asks whether this bound is approximately correct for $k$-term arithmetic progressions.

\begin{ques}[Cameron and Erd\H{o}s~{\cite[Section~4.2]{erdos-question}}]\label{ques: enumeration}
Is the number of subsets of $[n]$ not containing a $k$-term arithmetic progressions equal to $2^{(1+ o(1))r_k(n)}$? 
\end{ques} 

While Question~\ref{ques: enumeration} is still open, it is reasonable to believe that the answer could be yes. There are many results in the literature which assert that, up to a $(1+o(1))$ factor in the exponent, the number of combinatorial structures satisfying a certain set of constraints is equal to the number of subsets of the largest such structure. See, for example, the results of Kleitman~\cite{kleitman} on antichains, Dong, Mani, and Zhao~\cite{t-error} on $t$-error correcting codes, and Balogh, Das, Delcourt, Liu, and Sharifzadeh~\cite{delcourt} on intersecting families. The following result of Balogh, Liu, and Sharifzadeh~\cite{balogh} makes substantial progress towards answering Question~\ref{ques: enumeration} in the affirmative. 

\begin{thm}[{Balogh, Liu, and Sharifzadeh~\cite[Theorem~4.2]{balogh}}]\label{thm: kap-free}
For $k\geq3$ and for infinitely many $n\in\N$, the number of subsets of $[n]$ with no $k$-term arithmetic progression is $2^{O(r_k(n))}$. 
\end{thm}

Our main result is the following generalization of Theorem~\ref{thm: kap-free}. 

\newThm{main-result}
{Let $d$ be a positive integer and let $X \subseteq \N^d$ be a finite set such that $|X| \geq 3$. For infinitely many $n\in\N$, the number of $X$-free subsets of $[n]^d$ is $2^{O(r_X(n))}$.}

The statement of Theorem~\ref{thm:main-result} was recently proven by Kim~\cite{kim} in the special case that $X=\{\vec{0}, \vec{e}_1,\dots,\vec{e}_d\}$, where $\vec{e}_1,\dots,\vec{e}_d$ are the vectors of the standard basis of $\mathbb{R}^d$. An $X$-free set for this particular choice of $X$ is said to be \emph{corner-free}. It is well-known that the specialization of the Multidimensional Szemer\'edi Theorem to corner-free sets implies the full Multidimensional Szemer\'edi Theorem; see, e.g.,~\cite[Proof Theorem~10.3]{gowers-multi}. However, there seems to be no obvious reduction of Theorem~\ref{thm:main-result} to the result of Kim~\cite{kim} due to the fact that the function $r_X(n)$ for a generic set $X$ may have a very different growth rate when compared with the maximum size of a corner-free set.

As in~\cite{balogh} and~\cite{kim}, we will obtain our results by applying the seminal hypergraph container method, which was introduced by Balogh, Morris, and Samotij~\cite{morris} and, independently, Saxton and Thomason~\cite{saxton}. It is a widely applicable tool which has since been used to prove a myriad of enumeration results; see the survey~\cite{survey-containers}. 

In applying the hypergraph container method to prove Theorem~\ref{thm:main-result}, the high level idea is to show that the $X$-free subsets of $[n]^d$ can be covered by $2^{o(r_X(n))}$ sets called ``containers'', each of which only has $O(r_X(n))$ elements. As in most applications of the this method, the key to obtaining these containers is to prove a ``supersaturation'' result which ensures that any subset of $[n]^d$ with $C\cdot r_X(n)$ elements, for large $C$, contains many copies of $X$. Due to our limited understanding of the exact growth rate of $r_X(n)$ (and, in particular, how severely it ``fluctuates'' as $n$ grows), we are only able to obtain such a supersaturation result on an infinite subset of $\N$; this explains why Theorem~\ref{thm:main-result} does not apply to all $n\in\N$. 

The exact asymptotics of $r_X(n)$ have not been determined, even when $X$ is a 3-term arithmetic progression. In 1946, Behrend~\cite{behrend} presented a construction of large sets that do not contain $3$-term arithmetic progressions, giving the following lower bound on $r_3(n)$,
\[r_3(n) \geq \Omega\left(\frac{n}{2^{2\sqrt{2}\sqrt{\log_2 n}}\cdot \log^{1/4}(n)} \right). \]
In 2011, this bound was improved upon by a factor of $\Theta(\sqrt{\log n})$ by Elkin~\cite{elkin}. Shortly after Elkin's paper, Green and Wolf~\cite{wolf} found a simpler proof of the same bound. The first bound of $r_3(n)=o(n)$ was famously proved by Roth~\cite{roth} and has been improved many times throughout the years~\cite{bloom,meka,Sanders,BloomSisask19,Bloom16,Schoen}. Currently, the best known upper bound is $r_3(n) \leq n \cdot 2^{-O((\log n)^\beta)}$, for some absolute constant $\beta>0$, proved by Kelley and Meka~\cite{meka}. 

The case for $k > 3$ has also been studied extensively. Currently, the best known bounds on $r_k(n)$ are:
\[\frac{Cn\sqrt[2a]{\log n}}{2^{a2^{(a-1)/2}\sqrt[a]{\log n}}} \leq r_k(n) \leq \frac{n}{(\log \log n)^{C'}} , \]
where $a = \lceil\log k\rceil$ and $C,C'$ are positive constants depending only on $k$. The lower bound was obtained by O'Bryant~\cite{obryant} by generalizing the ideas of Elkin~\cite{elkin} and Green and Wolf~\cite{wolf} and incorporating them into an argument of Rankin~\cite{rankin}, who gave a  natural generalization of Behrend's construction~\cite{behrend}. The upper bound is due to Gowers~\cite{gowers-fourier}. 

%Jon: I've moved this comment nearer to the question and reworded it

%Not knowing the asymptotics of $r_k(n)$ is part of what makes Question~\ref{ques: enumeration} particularly difficult. Despite this, it is not unreasonable to think the answer to this question could be yes. Similar enumeration results have been proven true for other set properties, such as with antichains by Kleitman~\cite{kleitman}, $t$-error correcting codes by Dong, Mani, and Zhao~\cite{t-error}, and intersecting families by Balogh, Das, Delcourt, Liu, and Sharifzadeh~\cite{delcourt}.

The rest of the paper is organized as follows. In Section~\ref{sect:supersat-lemmas} we prove a supersaturation lemma for all $n$ using a consequence of the prime number theorem. In Section~\ref{sect:supersat}, we use the results of Section~\ref{sect:supersat-lemmas} together with a generalization of Behrend's construction to prove a stronger supersaturation result for an infinite subset of $\mathbb{N}$. Finally, in Section~\ref{sect:mainresult}, we will state a hypergraph container lemma and combine it with the supersaturation result from Section~\ref{sect:supersat} to prove Theorem~\ref{thm:main-result}. Throughout this paper, let $d$ be a positive integer and let $X$ be a finite subset of $\N^d$ such that $|X| \geq 3$. All logarithms in this paper are natural unless otherwise specified. 

\section{Supersaturation lemmas}\label{sect:supersat-lemmas}

As is typical with applications of the hypergraph container method, one of the ingredients used to prove Theorem~\ref{thm:main-result} is a supersaturation result. In this section, we will prove some preliminary supersaturation results that provide a lower bound for the number of copies of $X$ in a large set $A \subseteq [n]^d$. The results of this section will be used in Section~\ref{sect:supersat} to prove a stronger supersaturation result for an infinite set of positive integers. For $A\subseteq [n]^d$, let $\Gamma_X(A)$ be the number of copies of $X$ in $A$. We begin with the following simple lower bound on $\Gamma_X(A)$. 

\begin{lem}
\label{lem:supersat1}
For any $A\subseteq [n]^d$, 
\[\Gamma_X(A)\geq |A|-r_X(n).\]
\end{lem}

\begin{proof}
We proceed by induction on $|A|$. If $|A|\leq r_X(n)$, then the right side of the inequality is at most zero so there is nothing to prove. Now, assume $|A|\geq r_X(n) + 1$. Then $A$ must contain at least one copy of $X$, say $X'$. For any $x\in X'$, the set $A\setminus\{x\}$ contains at least $|A|-1-r_X(n)$ copies of $X$ by the induction hypothesis. These copies, along with $X'$ which is not in $A \setminus \{x\}$, yield $|A|-r_X(n)$ copies of $X$ in $A$. 
\end{proof}

The bound in Lemma~\ref{lem:supersat1} is, obviously, rather trivial. However, as we see in Lemma~\ref{lem:supersatPrime} below, we can obtain a seemingly stronger bound by applying Lemma~\ref{lem:supersat1} within copies of $[M]^d$ inside of $[n]^d$, where $1\leq M\leq n$, and using an ``averaging argument.'' The copies of $[M]^d$ that we consider are translates of $p\cdot [M]^d$ where $p$ is a small prime; the following consequence of the prime number theorem~(see, e.g.,~\cite{pnt}) is useful for counting the choices of $p$. For $\ell\geq2$, let $\pi(\ell)$ be the number of primes $p$ such that $2\leq p\leq \ell$.

%Jon: I decided that this paragraph could be more descriptive.

%Next, we will apply Lemma~\ref{lem:supersat1} and the following simple consequence of the Prime Number Theorem to prove a stronger supersaturation result. 

\begin{prop}
\label{prop:PNT}
There exists $\ell_0\in \mathbb{N}$ such that $\pi(\ell)\geq \frac{\ell}{2\log(\ell)}$ for all $\ell\geq \ell_0$.
\end{prop}

\begin{lem}
\label{lem:supersatPrime}
Let $M$ and $n$ be integers such that $2\leq M\leq n$. If $A\subseteq [n]^d$ such that 
\begin{equation}\label{eq:supersat3condition}|A|> \max\{4\ell_0dn^{d-1}M, r_X(n)\}\end{equation}
where $\ell_0$ is as in Proposition~\ref{prop:PNT}, then 
\[\Gamma_X(A)\geq \frac{|A|}{11d\log^2(n)}\cdot \frac{n}{M}\left(\frac{|A|}{2n^d} - \frac{r_X (M)}{M^d}\right).\]
\end{lem}

%Jon: Its good to include this observation in the thesis. But its probably too obvious to state in a paper. 

%Before proceeding with the proof of Lemma~\ref{lem:supersatPrime}, we make the following observation:

%\begin{obs}
%\label{obs:shift}
%If $X\subseteq \mathbb{N}^d$ and $X'$ is a non-trivial copy of $X$, then a set $A\subseteq [n]^d$ is $X$-free if and only if it is $X'$-free. 
%\end{obs}

%Therefore, if $X$ and $X'$ are non-trivial copies of each other, we may work with them interchangeably. 

%\begin{proof}[Proof of Lemma~\ref{lem:supersatPrime}]
\begin{proof}
By translating and scaling $X$ if necessary, we may assume that there does not exist any ``smaller'' non-trivial copy of $X$ in $\mathbb{N}^d$. Specifically, we can assume that there does not exist an integer $t>1$ which divides the coordinates of the vector $\vec{x}-\vec{y}$ for all $\vec{x},\vec{y}\in X$. 

Let $p$ be a prime chosen uniformly at random from all primes $p\leq \frac{|A|}{4dn^{d-1}M}$. Next, choose a ``base point'' $\vec{b}$ uniformly at random from the set $\{0,\dots,n-Mp-1\}^d$ (which is non-empty because $Mp\leq\frac{|A|}{4dn^{d-1}}<n$) and define
\[G_{\vec{b}} = \vec{b} + p\cdot[M]^d.\]
Let $A_{\vec{b}}:= A\cap G_{\vec{b}}$. Our goal is to bound $\mathbb{E}[\Gamma_X(A_{\vec{b}})]$ from below in terms of $|A|$ and from above in terms of $\Gamma_X(A)$; combining these inequalities will give us the desired lower bound on $\Gamma_X(A)$. 

Note that $G_{\vec{b}}$ is a non-trivial copy of $[M]^d$, so we can apply Lemma~\ref{lem:supersat1} within $G_{\vec{b}}$ to yield
\begin{equation}\label{eq:gamma-first-bound}
\mathbb{E}[\Gamma_X(A_{\vec{b}})] \geq \mathbb{E}[|A_{\vec{b}}| - r_X (M)] = \mathbb{E}[|A_{\vec{b}}|] - r_X (M)
\end{equation}
by linearity of expectation. So, to get a lower bound on $\mathbb{E}[\Gamma_X(A_{\vec{b}})]$, it suffices to bound $\mathbb{E}[|A_{\vec{b}}|]$. Define $I_p := \{0,\dots,n-Mp-1\}^d\cap [n]^d$; that is $I_p$ is obtained from $[n]^d$ by removing all points where some coordinate is `large'. As there are fewer than $dn^{d-1}Mp$ such points and $p\leq \frac{|A|}{4dn^{d-1}M}$, we have that
\[|A \cap I_p| > |A| - dn^{d-1}Mp \geq |A| - dn^{d-1}M\left(\frac{|A|}{4dn^{d-1}M}\right) = \frac{|A|}{2}\]
for every possible choice of $p$. Given the choice of $p$, for any $v\in I_p$ there are exactly $M^d$ choices of base point $\vec{b}\in \{0,\dots,n-Mp-1\}^d$ with the property that $v\in G_{\vec{b}}$. Thus, conditioned on the choice of $p$, the probability that any given $v\in I_p$ is in $G_{\vec{b}}$ is at least $\frac{M^d}{(n-Mp)^d}\geq \frac{M^d}{n^d}$. Therefore, $\mathbb{E}[|A_{\vec{b}}|] \geq \mathbb{E}[|A \cap I_p \cap G_{\vec{b}}|] \ge \frac{|A|}{2}\cdot\frac{M^d}{n^d}$. Substituting this lower bound into \eqref{eq:gamma-first-bound} gives
\begin{equation}\label{eq:exp-gamma-lb}
\mathbb{E}[\Gamma_X(A_{\vec{b}})] \geq \frac{|A|}{2}\cdot\frac{M^d}{n^d} - r_X (M).
\end{equation}

Next, we will obtain an upper bound on $\mathbb{E}[\Gamma_X(A_{\vec{b}})]$ in terms of $\Gamma_X(A)$. To do this, we let $Y$ be an arbitrary fixed non-trivial copy of $X$ in $A$ and bound the probability that $Y\subseteq G_{\vec{b}}$ from above. Let $\vec{a}\in \mathbb{R}^d$ and $r>0$ so that $Y=\vec{a}+r\cdot X$ and note that, since $|X|\geq2$, the choice of $\vec{a}$ and $r$ is unique. Then, clearly, $Y$ is a subset of $G_{\vec{b}}$ if and only if $(\vec{a}-\vec{b})+r\cdot X$ is a subset of $p\cdot [M]^d$. Thus, if $Y$ is a subset of $G_{\vec{b}}$, then, for any $\vec{x},\vec{y}\in X$, every coordinate of 
\[(\vec{a}-\vec{b}+r\vec{x}) -(\vec{a}-\vec{b}+r\vec{y}) = r(\vec{x}-\vec{y})\]
is a multiple of $p$. By our assumption that $X$ is the smallest copy of itself in $\N^d$, this implies that $r$ is an integer and $p$ divides $r$. Thus, for any non-trivial copy $Y=\vec{a}+rX$ of $X$ in $A$, the probability that $Y$ is a subset of $G_{\vec{b}}$ is at most the probability that $p$ is a divisor of $r$, multiplied by the probability that, for an arbitrary $\vec{y}_0\in Y$, the base point $\vec{b}$ is chosen to be one of the at most $M^d$ base points for which $\vec{y}_0\in G_{\vec{b}}$. Let us now bound these two probabilities individually. Since, for any such $Y$, the integer $r$ has at most $\log_2(r)\leq \log_2(n)$ divisors, the probability that $p$ divides $r$ is at most
\[\frac{\log_2(n)}{\pi(|A|/4dn^{d-1}M)}\leq \frac{\log_2(n)8dn^{d-1}M\log(|A|/4dn^{d-1}M)}{|A|}\leq \frac{8\log^2_2(n)\cdot Mdn^{d-1}}{|A|}\]
where the first inequality uses Proposition~\ref{prop:PNT} and the fact that $|A|>4\ell_0dn^{d-1}M$. After fixing the choice of $p$, for any $\vec{y}_0\in Y$, the probability that $\vec{y}\in G_{\vec{b}}$ is at most
\[\frac{M^d}{(n-Mp)^d}\leq \frac{M^d}{\left(n-\frac{|A|}{4dn^{d-1}}\right)^d}.\]
Observe that
\[\left(\frac{M}{n- \frac{|A|}{4dn^{d-1}}}\right)^d = \left(\frac{M}{n}\right)^d\left({1 - \frac{|A|}{4dn^d}}\right)^{-d} \leq \left(\frac{M}{n}\right)^d\left(1 - \frac{1}{4d}\right)^{-d} \leq \left(\frac{M}{n}\right)^d\cdot\frac{4}{3}. \]
Therefore, for any non-trivial copy $Y$ of $X$ in $A$, we have
\[\mathbb{P}[Y\subseteq A_{\vec{b}}] \leq  \frac{8\log_2^2(n) \cdot Mdn^{d-1}}{|A|}\cdot \left(\frac{M}{n}\right)^d\cdot\frac{4}{3} \leq \frac{11\cdot\log_2^2(n)M^{d+1}d}{|A|n}. \]
Using this, we get the following upper bound on $\mathbb{E}[\Gamma_X(A_{\vec{b}})]$ by linearity of expectation:
\begin{equation}\label{eq:exp-gamma-ub}
\mathbb{E}[\Gamma_X(A_{\vec{b}})] \leq \Gamma_X(A) \cdot\frac{11\cdot\log_2^2(n)M^{d+1}d}{|A|n}. 
\end{equation}
Putting together \eqref{eq:exp-gamma-lb} and \eqref{eq:exp-gamma-ub} and solving for $\Gamma_X(A)$ gives,
\[\Gamma_X(A) \geq \frac{n |A|}{11\log_2^2(n)M^{d+1}d}\left(\frac{|A| M^d}{2n^d} - r_X (M)\right)= \frac{|A|}{11d\log^2(n)}\cdot \frac{n}{M}\left(\frac{|A|}{2n^d} - \frac{r_X (M)}{M^d}\right),\]
as desired.
\end{proof}

\section{Stronger supersaturation on an infinite set}\label{sect:supersat}

Our aim in this section is to prove a stronger supersaturation result for infinitely many values of $n$. We begin by proving the following basic lemma about about functions of $n$ with a particular growth rate. 

\begin{lem}
\label{lem:sequence}
For $\beta_r,\beta_m>0$, let $0<\alpha<\exp(-\beta_r\beta_m/2)$. For $d\geq1$, suppose that $r:\mathbb{N}\to \mathbb{N}$ and $m:\mathbb{N}\to\mathbb{N}$ are such that
\[r(n)\geq \frac{n^d}{\exp(\beta_r\sqrt{\log(n)})}\]
and
\[n \geq m(n)\geq \frac{n}{\exp(\beta_m\sqrt{\log(n)})}\]
for all $n\geq1$. Then there exists an infinite subset $N$ of $\mathbb{N}$ such that
\[\frac{ r(n)}{n^d}\geq \frac{\alpha\cdot r(m(n))}{m(n)^d}\]
for all $n\in N$. 
\end{lem}

\begin{proof}
Take $\varepsilon>0$ so that $\alpha = \frac{\exp\left(\frac{-\beta_r\beta_m}{2}\right)}{(1+\varepsilon)^{3}}$. By definition of $m(n)$, 
\begin{align}
\sqrt{\log(m(n))} - \sqrt{\log(n)} &\geq \sqrt{\log\left(\frac{n}{\exp(\beta_m\sqrt{\log(n)})}\right)} - \sqrt{\log(n)}\nonumber \\
&= \sqrt{\log(n) - \beta_m\sqrt{\log(n)}} - \sqrt{\log(n)}. \label{eq:sq-conv}
\end{align}
Observe that the expression in \eqref{eq:sq-conv} converges to $\frac{-\beta_m}{2}$ as $n \to \infty$. Therefore, there exists some $n_0 \in \N$ such that, for all $n \geq n_0$, \[\sqrt{\log(m(n))} - \sqrt{\log(n)} \geq \frac{-\beta_m}{2} - \frac{\log(1+\varepsilon)}{\beta_r}. \] 
This implies that for all $n \geq n_0$, \begin{equation}\label{eq:exp-LB}
\exp{\left(\beta_r\left(\sqrt{\log(m(n))} - \sqrt{\log(n)}\right)\right)} \geq \frac{\exp{(\frac{-\beta_r\beta_m}{2})}}{(1+\varepsilon)}. 
\end{equation}

Define 
\[Y:= \left\{ \frac{r(n) \exp(\beta_r \sqrt{\log(n)})}{n^d}: n \in \N \right\}. \] Let $Z$ be the set of all cluster points of $Y$ in the extended reals, $\R \cup \{-\infty, \infty\}$. Note that, by the lower bound on $r(n)$, we have $Z \subseteq [1, \infty]$. Let $z = \sup(Z)$. We now proceed by cases depending on the value of $z$. 

Suppose $z = \infty$. Define \[N := \left\{n \in \N: n \geq n_0 \text{ and } \frac{r(n)\exp(\beta_r\sqrt{\log(n)})}{n^d} \geq \frac{r(m)\exp(\beta_r\sqrt{\log(m)})}{m^d}, \forall m \leq n \right\}. \]
As $z = \infty$, it follows that $N$ is an infinite set. As $m(n) \leq n$, for all $n \in N$ we have, 
\begin{align*}
\frac{r(n)}{n^d} \geq \frac{r(m(n))\cdot\exp(\beta_r\sqrt{\log(m(n))})}{m(n)^d \cdot\exp(\beta_r\sqrt{\log(n)})} \geq \frac{r(m(n))}{m(n)^d}\cdot\frac{\exp(\frac{-\beta_m\beta_r}{2})}{(1+\varepsilon)} \geq \alpha \cdot \frac{r(m(n))}{m(n)^d},
\end{align*}
where the second inequality follows from \eqref{eq:exp-LB}. Therefore, $N$ is as desired. 

We may now assume that $z < \infty$. Choose $n_0'$ large enough such that for all $n \geq n_0'$, 

\begin{equation}\label{eq:rmn-ub}
r(m(n))\cdot \frac{\exp(\beta_r\sqrt{\log(m(n))})}{m(n)^d} \leq (1+\varepsilon)z. 
\end{equation}
Note that such an $n_0'$ exists as $m(n) \to \infty$ as $n \to \infty$. Define
\[N := \left\{n \in \N: n \geq \max\{n_0,n_0'\} \text{ and } \frac{r(n) \cdot \exp(\beta_r\sqrt{\log(n)})}{n^d} \geq \frac{z}{(1+\varepsilon)}\right\}. \]
By definition of $z$, the set $N$ is infinite. By definition of $N$, for all $n \in N$,
\begin{align*}
\frac{r(n)}{n^d} &\geq \frac{z}{(1+\varepsilon)\cdot \exp({\beta_r\sqrt{\log(n)}})} \stackrel{\eqref{eq:rmn-ub}}{\geq} \frac{r(m(n))}{m(n)^d}\cdot \frac{\exp({\beta_r\sqrt{\log(m(n))}})}{\exp({\beta_r\sqrt{\log(n)}})\cdot (1+\varepsilon)^2}\\
&\stackrel{\eqref{eq:exp-LB}}{\geq} \frac{\exp({-\frac{\beta_r\beta_m}{2}})}{(1+\varepsilon)^3}\cdot\frac{r(m(n))}{m(n)^d} = \alpha\cdot\frac{r(m(n))}{m(n)^d}.
\end{align*}
Therefore, $N$ is as desired. 
\end{proof}

In proving the main result of this section, we will apply Lemma~\ref{lem:sequence} with $r(n) = r_X(n)$. To do so, we require the following standard extension of the  construction of Behrend~\cite{behrend} to higher dimensions. For completeness, the proof is included in Appendix~\ref{appendix-A}.

\newLem{BehrendAllDim}{Let $d\geq1$ and $X\subseteq \mathbb{N}^d$ be fixed. If $|X|\geq3$, then there exists  $c_X>0$ depending only on $X$ such that
\[r_X(n) \geq n^d\cdot \exp({-c_X\sqrt{\log{n}}}).\]
} 
In order to apply Lemma~\ref{lem:sequence}, we require another function $m(n)$. Define, for the rest of this section,
\[m(n) := \frac{n}{\log^{3|X|}(n)}\left(\frac{r_X(n)}{n^d}\right)^{|X|+2}. \] 
Using Lemma~\ref{lem:BehrendAllDim}, observe that, for sufficiently large $n$, 
\begin{align*}
m(n) &= \frac{n}{\log^{3|X|}(n)}\left(\frac{r_X(n)}{n^d}\right)^{|X|+2} \geq \frac{n}{\log^{3|X|}(n)}\left( \exp({-c_X\sqrt{\log n})}\right)^{|X|+2} \\
&>n \cdot\left( \exp\left({-(5c_X|X|)\sqrt{\log n}}\right)\right).\stepcounter{equation}\tag{\theequation}\label{eq:m(n)-bound}
\end{align*}
Let $\alpha$ be chosen such that $0 < \alpha < \exp\left(-\frac{5c_X^2|X|}{2}\right)$.  This choice of $\alpha$ is fixed for the remainder of this paper. 

\begin{lem}\label{lem:supersat-seq}
There exists an infinite subset $N$ of $\mathbb{N}$ such that, if $n\in N$ and $A\subseteq [n]^d$ where $|A|\geq \frac{4}{\alpha}\cdot r_X(n)$, then 
\[\Gamma_X(A) \geq \frac{\log^{3|X|-2}(n)}{3d}\cdot\left(\frac{n^d}{r_X(n)}\right)^{|X|}\cdot n^d.\]
\end{lem}

\begin{proof}

Let $\ell_0$ be as defined in Proposition~\ref{prop:PNT}. By Lemma~\ref{lem:BehrendAllDim} and \eqref{eq:m(n)-bound}, we may apply Lemma~\ref{lem:sequence} with $m(n)$, $r_X(n)$, and $\alpha$ to obtain an infinite subset $N'$ of $\mathbb{N}$, such that for all $n \in N'$,
\begin{equation}\label{eq:seq-bound}
\frac{r_X(n)}{\alpha\cdot n^d} \geq\frac{r_X(m(n))}{m(n)^d}. 
\end{equation}
Note that $\alpha<1$ and so the condition $|A|\geq \frac{1}{\alpha}\cdot r_X(n)$ implies that $|A|\geq r_X(n)$. For large enough $n\in N'$, we have
\[4\ell_0dn^{d-1}m(n) = 4\ell_0dn^{d-1}\cdot\frac{n}{\log^{3|X|}(n)}\cdot\left(\frac{r_X(n)}{n^d}\right)^{|X|+2} 
\leq 4\ell_0dn^{d}\left(\frac{r_X(n)}{n^d}\right)^2\]
\[= 4\ell_0d\left(\frac{r_X(n)}{n^d}\right)\cdot r_X(n)
< \frac{4}{\alpha}\cdot r_X(n)\leq |A|,
\]
where the penultimate inequality holds for sufficiently large $n$ as $r_X(n) = o(n^d)$.
Therefore $A$ satisfies the conditions of Lemma~\ref{lem:supersatPrime} with $M=m(n)$. Applying this we obtain 
\begin{align*}
\Gamma_X(A) &\geq \frac{1}{11d\log^2(n)}\cdot |A| \cdot \frac{n}{m(n)}\cdot\left(\frac{|A|}{2n^d} - \frac{r_X(m(n))}{m(n)^d}\right). 
\end{align*}
Using $|A|\geq \frac{4}{\alpha}\cdot r_X(n)$ and using \eqref{eq:seq-bound} to bound the final term, we get
\begin{align*}
\Gamma_X(A) &\geq \frac{1}{11d\log^2(n)}\cdot |A| \cdot \frac{n}{m(n)}\cdot\left(\frac{2r_X(n)}{\alpha \cdot n^d} - \frac{r_X(n)}{\alpha\cdot n^d}\right) \\ 
&= \frac{1}{11d\log^2(n)}\cdot |A| \cdot \frac{n}{m(n)} \cdot\left(\frac{r_X(n)}{\alpha\cdot n^d}\right).
\end{align*}
Further, we may apply the definition of $m(n)$ and the lower bound on $|A|$ to get
\begin{align*}
\Gamma_X(A) &\geq \frac{1}{11d\log^2(n)}\cdot |A| \cdot n \cdot \frac{\log^{3|X|}(n)}{n}\cdot \left(\frac{n^d}{r_X(n)} \right)^{|X|+2}\cdot\left(\frac{r_X(n)}{\alpha\cdot n^d}\right) \\
&\geq \frac{1}{11d\log^2(n)}\cdot \frac{4}{\alpha}\cdot r_X(n) \cdot \log^{3|X|}(n)\cdot \left(\frac{n^d}{r_X(n)} \right)^{|X|+2}\cdot\left(\frac{r_X(n)}{\alpha\cdot n^d}\right) \\
&\geq  \frac{\log^{3|X|-2}(n)}{3d\alpha^2}\cdot\left(\frac{n^d}{r_X(n)}\right)^{|X|}\cdot n^d. 
\end{align*}
Hence, as $\alpha < 1$, we have for large enough $n \in N'$
\begin{equation}\label{eq:supersat-result}
\Gamma_X(A) \geq \frac{\log^{3|X|-2}(n)}{3d}\cdot\left(\frac{n^d}{r_X(n)}\right)^{|X|}\cdot n^d. 
\end{equation}
Define $N$ to be the set of all $n\in N'$ which are large enough that \eqref{eq:supersat-result} holds. 
\end{proof}

\section{Main result}\label{sect:mainresult}

In this section, we prove Theorem~\ref{thm:main-result}. The proof applies a corollary of a hypergraph container lemma of Saxton and Thomason~\cite{saxton}. To state this lemma, we require some definitions.  

Consider an $r$-uniform hypergraph $\mathcal{H} = (V(\mathcal{H}), E(\mathcal{H}))$, where $V(\mathcal{H})$ is the set of vertices and $E(\mathcal{H})$ is the set of hyperedges. We use $\bar{d}(\mathcal{H})$ to denote the average degree of a vertex in $\mathcal{H}$; by a handshaking argument, $\bar{d}(\mathcal{H})=r|E(\mathcal{H})|/|V(\mathcal{H})|$. For a set $A\subseteq V(\mathcal{H})$, we use $e_\mathcal{H}(A)$ to denote the number of hyperedges contained in $A$. We say $A \subseteq V(\mathcal{H})$ is \emph{independent} if $e_\mathcal{H}(A) = 0$. The \emph{co-degree} of a set $A \subseteq V(\mathcal{H})$ is the number of hyperedges of $\mathcal{H}$ that contain $A$. We use $d_{\mathcal{H}}(A)$, or simply $d(A)$ when the subscript is obvious from context, to denote the co-degree of a set $A$.  We denote the maximum co-degree of all sets of size $k$ in $\mathcal{H}$ as $\Delta_k(\mathcal{H})$; explicitly \[\Delta_k(\mathcal{H}) := \max\{d_{\mathcal{H}}(A): |A| = k\}.\] Finally, we define $\Delta(\mathcal{H}, \tau)$ for any $\tau > 0$ as \[\Delta(\mathcal{H}, \tau) := 2^{\binom{r}{2} - 1} \sum_{j=2}^r 2^{-\binom{j-1}{2}} \frac{\Delta_j(\mathcal{H})}{\tau^{j-1}\bar{d}(\mathcal{H})}. \] Notice that $\Delta(\mathcal{H}, \tau)$ depends on all maximum co-degrees in $\mathcal{H}$, except for $\Delta_1(\mathcal{H})$. 

\begin{thm}[Hypergraph Container Lemma~\cite{saxton}]\label{thm:saxtonthomason}Let $\mathcal{H}$ be an $r$-uniform hypergraph. Suppose that $0 < \varepsilon, \tau < 1/2$  satisfy
\begin{itemize}
    \item $\tau < 1/(200\cdot r \cdot r!^2)$, and
    \item $\Delta(\mathcal{H}, \tau) \leq \varepsilon/(12r!).$
\end{itemize} 
Then there exists $c = c(r) \leq 1000\cdot r \cdot r!^3$ and a collection $\mathcal{C}$ of subsets of $V(\mathcal{H})$ such that the following hold:

\begin{itemize}
    \item for every independent set $I$ in $\mathcal{H}$, there exists $A \in \mathcal{C}$ such that $I \subseteq A$,  
    \item $\log|\mathcal{C}| \leq c \cdot |V(\mathcal{H})| \cdot \tau \cdot \log(1/\varepsilon) \cdot \log(1/\tau)$,
    \item $e_{\mathcal{H}}(A) \leq \varepsilon \cdot |E(\mathcal{H})|$ for every $A \in \mathcal{C}$. 
\end{itemize}
\end{thm}

 We will adapt Theorem~\ref{thm:saxtonthomason} to our problem. In order to do so, we must view our problem from the perspective of independent sets in a hypergraph. Let us consider the $|X|$-uniform hypergraph $\mathcal{G}$ with $V(\mathcal{G}) = [n]^d$ and hyperedge set consisting of all copies of $X$ in $[n]^d$. Clearly independent sets in $\mathcal{G}$ are precisely $X$-free subsets of $[n]^d$. Applying Theorem~\ref{thm:saxtonthomason} to this hypergraph $\mathcal{G}$, we immediately get the following corollary.

\begin{cor}\label{cor:container} For a finite set $X \subseteq \mathbb{N}^d$, let $\mathcal{G}$ be the $|X|$-uniform hypergraph on $n^d$ vertices encoding the set of all copies of $X$ in $[n]^d$. Suppose that there exist $0 < \varepsilon$, $\tau < 1/2$ such that

\begin{itemize}
    \item $\tau < 1/(200\cdot |X| \cdot |X|!^2)$   
    \item $\Delta(\mathcal{G}, \tau) \leq \varepsilon/(12|X|!).$
\end{itemize} 
Then there exists $c = c(|X|) \leq 1000\cdot |X| \cdot |X|!^3$ and a collection $\mathcal{C}$ of subsets of $V(\mathcal{G})$ such that the following holds:

\begin{itemize}
    \item for every $X$-free subset $I \subseteq [n]^d$, there exists $A \in \mathcal{C}$ such that $I \subseteq A$,  
    \item $\log|\mathcal{C}| \leq c \cdot n^d \cdot \tau \cdot \log(1/\varepsilon) \cdot \log(1/\tau)$,
    \item $e_{\mathcal{G}}(A) \leq \varepsilon \cdot |E(\mathcal{G})|$ for every $A \in \mathcal{C}$. 
\end{itemize}
\end{cor}

%Jon: I didn't see the need for a subsection
%\subsection{Proof of Theorem~\ref{thm:main-result}}

By combining Lemma~\ref{lem:supersat-seq} and Corollary~\ref{cor:container} we can now prove Theorem~\ref{thm:main-result}, restated below for convenience. 

\repeatThm{main-result}

\begin{proof}
Define $N$ to be the infinite subset of $\mathbb{N}$ guaranteed by Lemma~\ref{lem:supersat-seq}. Consider the $|X|$-uniform hypergraph $\mathcal{G}$ encoding copies of $X$ in $[n]^d$ and choose $\gamma>0$ so that $\bar{d}(\mathcal{G})= \frac{|X||E(\mathcal{G})|}{n^d} \geq \gamma n$ for all $n$ sufficiently large. For $n\in N$, define 
\[\varepsilon(n) := \frac{1}{3d}\cdot\frac{\log^{3|X|-2}(n)}{n}\cdot\left(\frac{n^d}{r_X(n)}\right)^{|X|}, \]
and
\[\tau(n):=\left(\frac{12|X|!\cdot 2^{|X|^2}\cdot |X|^3}{\gamma n\cdot \varepsilon(n)}\right)^{\frac{1}{|X|-1}}.\]
We wish to apply Corollary~\ref{cor:container} and so we will verify the conditions for $\mathcal{G}$ with $\varepsilon(n)$, and $\tau(n)$. Clearly, $\varepsilon(n), \tau(n) > 0$ for all $n$. Using Lemma~\ref{lem:BehrendAllDim}, observe for sufficiently large $n$,
\begin{align*}
\varepsilon(n) &= \frac{1}{3d}\cdot\frac{\log^{3|X|-2}(n)}{n}\cdot\left(\frac{n^d}{r_X(n)}\right)^{|X|} \\
&\leq \frac{1}{3d}\cdot\frac{\log^{3|X|-2}(n)}{n}\cdot\left(\frac{n^d}{n^d\cdot\exp(-c_X\sqrt{\log(n)})}\right)^{|X|}\\
&= \frac{1}{n^{1-o(1)}} < \frac{1}{2}. \stepcounter{equation}\tag{\theequation}\label{eq:eps-small}
\end{align*}
Also, we clearly have $\frac{1}{\varepsilon(n)}=o(1)$ and so $\tau(n)=o(1)$. Thus, for $n$ sufficiently large, we have
\begin{equation}\label{eq:taucondition}
\tau(n) < 1/(200\cdot |X| \cdot |X|!^2) < \frac{1}{2}. 
\end{equation}

Next, we wish to bound $\Delta(\mathcal{G}, \tau(n))$. By definition,
\begin{align*}
\Delta(\mathcal{G}, \tau(n)) = 2^{\binom{|X|}{2} - 1} \sum_{j=2}^{|X|} 2^{-\binom{j-1}{2}} \frac{\Delta_j(\mathcal{G})}{\tau(n)^{j-1}\bar{d}(\mathcal{G})}.
\end{align*}
Every term in the above summation is at most $\frac{\Delta_2(\mathcal{G})}{\tau(n)^{|X|-1}\bar{d}(\mathcal{G})}\leq \frac{|X|^2}{\tau(n)^{|X|-1}\gamma n}$ and so
\begin{equation}\label{eq:codegreecondition}
\Delta(\mathcal{G},\tau(n)) \leq  \frac{2^{|X|^2}\cdot |X|^3}{\tau(n)^{|X|-1}\gamma n} = \frac{\varepsilon(n)}{12|X|!}\end{equation}
where the equality is by definition of $\tau(n)$.

From \eqref{eq:eps-small}, \eqref{eq:taucondition} and \eqref{eq:codegreecondition}, we may apply Corollary~\ref{cor:container} with $\varepsilon = \varepsilon(n)$ and $\tau = \tau(n)$ to get a constant $c = c(X) \leq 1000\cdot |X| \cdot |X|!^3$ and a collection $\mathcal{C}$ of subsets of $V(\mathcal{G})$ with the following properties:
\begin{itemize}
    \item for every $X$-free subset $I \subseteq [n]^d$, there exists $S \in \mathcal{C}$ such that $I \subseteq S$,  
    \item $\log|\mathcal{C}| \leq c \cdot n^d \cdot \tau(n) \cdot \log(1/\varepsilon(n)) \cdot \log(1/\tau(n))$,
    \item for every $A \in \mathcal{C}$, $e_\mathcal{G}(A) \leq \varepsilon(n) \cdot |E(\mathcal{G})|.$
\end{itemize}

By definition, $\log(1/\varepsilon(n)) = O(\log(n))$ and $\log(1/\tau(n))=O(\log(n))$. We also have $\frac{n^d \log^3 (n)}{r_X(n)} = o\left((n\cdot \varepsilon(n))^{\frac{1}{|X|-1}}\right)$ which implies that $\tau(n)=o\left( \frac{r_X(n)}{n^d\log^3(n)}\right)$. Therefore,
\begin{equation}\label{eq:littleo}\log|\mathcal{C}|=O\left(n^d\tau(n)\log^2(n)\right) = o(r_X(n)).\end{equation}

%Observe that \begin{align}
 %       \log\frac{1}{\varepsilon(n)}\cdot\log\frac{1}{\tau(n)} & = \log\left(\frac{n}{\frac{3}{8^d\cdot d}\cdot\log^{3|X|-2}n\cdot \left(\frac{n^d}{r_X(n)}\right)^{|X|-1}}\right)\cdot \log\left(\frac{n^d\cdot \log^3n}{r_X(n)}\right)\nonumber\\
  %      &\leq \log(n)\cdot\log(n^{d+3})\nonumber\\
   %     & = (d+3)\log^2n. \label{eq:logs}
    %\end{align} 
    
%Using \eqref{eq:logs} and the definition of $c$, 
 %   \begin{align}
 %       \log|\mathcal{C}| & \leq c\cdot n^d\cdot\tau(n)\cdot\log\frac{1}{\varepsilon(n)}\cdot\log\frac{1}{\tau(n)}\nonumber\\
  %      & \leq 1000|X||X|!^3\cdot n^d \cdot\log\frac{1}{\varepsilon(n)}\cdot\log\frac{1}{\tau(n)}\nonumber\\
   %     & \overset{\eqref{eq:logs}}{\leq} 1000|X||X|!^3\cdot n^d \cdot\frac{r_X(n)}{n^d}\cdot\frac{1}{\log^3(n)}
    %    \cdot (d+3)\log^2n\nonumber\\
     %   & = o(r_X(n)).\label{eq:littleo}
    %\end{align}

Note that for every container $A \in \mathcal{C}$, the number of copies of $X$ in $A$ is at most 
$\varepsilon(n)\cdot |E(\mathcal{G})| < \varepsilon(n)\cdot n^{d+1}$. That is,
\[\Gamma_X(A)\leq \varepsilon(n)\cdot n^{d+1} = \frac{\log^{3|X|-2}(n)}{3d}\cdot\left(\frac{n^d}{r_X(n)}\right)^{|X|}\cdot n^d.\]
So, by Lemma~\ref{lem:supersat-seq}, every container $A \in \mathcal{C}$ must satisfy 
\begin{equation}\label{eq:acrxn}|A|<\frac{4}{\alpha}\cdot r_X(n)\end{equation} 
where $\alpha$ is a constant depending only on $X$. Since every $X$-free subset of $[n]^d$ is contained in some container in $\mathcal{C}$, we can use \eqref{eq:littleo} and \eqref{eq:acrxn} to conclude that the number of $X$-free subsets of $[n]^d$ is at most \begin{align*}\sum_{A\in \mathcal{C}}2^{|A|} & \leq |\mathcal{C}|\cdot\max_{A\in \mathcal{C}}2^{|A|}
\overset{\eqref{eq:littleo},\eqref{eq:acrxn}}{<} 2^{o(r_X(n))}\cdot2^{\frac{4}{\alpha}\cdot r_X(n)} = 2^{O(r_X(n))},\end{align*} as desired.
\end{proof} 
\bibliographystyle{amsplain}
\bibliography{containers-project-notes}

\appendix

\section{Behrend-type constructions in every dimension}\label{appendix-A}

The purpose of this appendix is to prove the following extension of the classical result of Behrend~\cite{behrend} to arbitrary sets of cardinality at least three in any dimension. The argument in the proof is standard (similar ideas are used in, e.g.,~\cite[Proof of Lemma~3.3]{kim} and~\cite[Proof of Theorem~10.3]{gowers-multi}), 
and we only present it for the sake of completeness.

\repeatLem{BehrendAllDim}

Before proceeding with the proof of Lemma~\ref{lem:BehrendAllDim}, it will be useful to deal with the case $d=1$ separately. This proof follows the standard argument of Behrend~\cite{behrend}. 

\begin{lem}
\label{lem:Behrendd=1}
If $X\subseteq \mathbb{N}$ such that $|X|\geq 3$, then there exists $c_X>0$ depending only on $X$ such that, for all $n\in\N$,
\[r_X(n) \geq n\cdot \exp({-c_X\sqrt{\log{n}}}).\]
\end{lem}

%Jon: The old proof of this seemed a bit too long to me. I've tried to do it a bit more succinctly. 

\begin{proof}
Note that $r_X(n)\geq r_{X'}(n)$ for any subset $X'$ of $X$; thus, we can assume that $|X|=3$. Without loss of generality, $X=\{a,b,c\}$ for $a<b<c$. Let $q>0$ be chosen so that 
\[(b-a) = q\cdot(c-b).\]
For $n\geq1$, define $N:=\left\lceil \sqrt{\log(n)}\right\rceil$ and $M=\left\lfloor\frac{n^{1/N}}{q+1}\right\rfloor$. For each vector $\vec{x}=(x_1,\dots,x_N)\in [M]^N$, the square of the Euclidean norm, namely $x_1^2+x_2^2+\cdots +x_N^2$, is an integer in $[N,N\cdot M^2]$. So, by the pigeonhole principle, there exists a choice of $r$ so that $r^2$ is an integer in $[N,N\dot M^2]$ and $S$ defined to be the intersection of the sphere of radius $r$ with $[M]^d$ satisfies 
\[|S|\geq \frac{M^N}{NM^2-N+1}\geq \frac{M^{N-2}}{N}.\]
By definition of $M$ and $N$, this is at least
\[\frac{1}{N}\left(\frac{n^{1/N}}{2(q+1)}\right)^{N-2}=\frac{n}{n^{2/N}N(2(q+1))^{N-2}}\geq n\cdot \exp(-c_X\sqrt{\log(n)}),\]
where the constant $c_X$ depends only on $q$. We fix $S$ for the rest of the proof.

Next, define $f:\mathbb{Z}^N\to \mathbb{Z}$ according to
\[f(x_1,\dots,x_N)=\sum_{i=1}^Nx_i((q+1)M)^{i-1}\]
We define $A:=f(S)$. Our goal is to show that $A$ is an $X$-free subset of $[n]$, at which point we will be done since we have already shown that $|S|\geq n\cdot\exp(-c_X\sqrt{\log(n)})$. Let us start by showing that $A\subseteq [n]$. Clearly every element of $A$ is an integer. Note that, for every $(x_1,\dots,x_N)\in [M]^N$, we have that $f(x_1,\dots,x_N)$ is a positive integer which is bounded above by 
\[M\cdot\sum_{i=1}^N((q+1)M)^{i-1} = M\cdot\left(\frac{(q+1)^NM^N-1}{(q+1)M-1}\right).\]
If $n$ is chosen large with respect to $X$, then we may assume that $qM>1$ and so the above expression is less than $(q+1)^NM^N$ which, by definition of $M$, is at most $n$. Note that we can freely assume that $n$ is large with respect to $X$ since, for any fixed $n_0$, we can pick $c_X$ large enough that $n^d\exp(-c_X\sqrt{\log(n)})<1\leq r_X(n)$ for all $n<n_0$, and so it suffices to consider $n\geq n_0$. Therefore, $A\subseteq [n]$.

Next, let us show that $|A|=|S|$. For this, we show that $f$ is injective on $[M]^N$. Let $\vec{x}=(x_1,\dots,x_N)$ and $\vec{y}=(y_1,\dots,y_N)$ be distinct elements of $[M]^N$ and let $j$ be the largest index such that $x_j\neq y_j$. Then, by the triangle inequality, 
\[\left|f(\vec{x})-f(\vec{y})\right| \geq |x_j-y_j|((q+1)M)^{j-1} - \sum_{i=1}^{j-1}|x_i-y_i|((q+1)M)^{i-1}\]
\[\geq ((q+1)M)^{j-1} - M\left(\frac{(q+1)^{j-1}M^{j-1}-1}{(q+1)M-1}\right)>0\]
which implies that $f(\vec{x})\neq f(\vec{y})$. Therefore, $|A|=|S|$. 

Finally, we show that $A$ is $X$-free. If not, then there must exist $\vec{x},\vec{y}$ and $\vec{z}$ in $S$ such that $f(\vec{x})<f(\vec{y})<f(\vec{z})$ and $f(\vec{y})-f(\vec{x}) = q(f(\vec{z})-f(\vec{y}))$. In other words, $(q+1)f(\vec{y})-f(\vec{x})-qf(\vec{z}) = 0$. Since $S$ is a subset of a sphere, $S$ is $X$-free and so it cannot be the case that $(q+1)\vec{y}-\vec{x}-q\vec{z}=\vec{0}$, and so we can let $j$ be the largest index such that $(q+1)y_j-x_j-qz_j\neq 0$. Now, by the triangle inequality,
\[|(q+1)f(\vec{y})-f(\vec{x})-qf(\vec{z})|\]
\[\geq |(q+1)y_j-x_j-qz_j|((q+1)M)^{j-1} - \sum_{i=1}^{j-1}|(q+1)y_i-x_i-qz_i|((q+1)M)^{i-1}\]
\[\geq ((q+1)M)^{j-1} -((q+1)M - q-1)\left(\frac{(q+1)^{j-1}M^{j-1}-1}{(q+1)M-1}\right)>0\]
which contradicts the assumption that $(q+1)f(\vec{y})-f(\vec{x})-qf(\vec{z}) = 0$. Therefore $A$ is $X$-free and the proof is complete. 
\end{proof}

We now use Lemma~\ref{lem:Behrendd=1} to prove Lemma~\ref{lem:BehrendAllDim}. 

\begin{proof}[Proof of Lemma~\ref{lem:BehrendAllDim}]
If $d=1$, then we are done by Lemma~\ref{lem:Behrendd=1}, and so we assume that $d\geq2$. Since $r_X(n)\geq r_{X'}(n)$ for any subset $X'$ of $X$, we may assume that $X=\{\vec{x},\vec{y},\vec{z}\}$. Our goal is to construct, for each $n\geq1$, a ``large'' set $S\subseteq[n]^d$ which is $X$-free. We may assume throughout the proof that $n$ is divisible by $4(d-1)$. 

We start by making numerous (valid) assumptions about the vectors $\vec{x},\vec{y}$ and $\vec{z}$. Let $T$ be the convex hull of $\{\vec{x},\vec{y},\vec{z}\}$ and note that $T$ is a (possibly degenerate) triangle. By relabelling $\vec{x},\vec{y}$ and $\vec{z}$ if necessary, we can assume that neither of the interior angles of $T$ at $\vec{x}$ nor $\vec{y}$ are equal to $\pi/2$; this is because any triangle can contain at most one interior angle of $\pi/2$. Note that, in the special case that $\vec{x},\vec{y}$ and $\vec{z}$ are collinear, $T$ is a degenerate triangle and the interior angles at $\vec{x}$ and $\vec{y}$ are either $0$ or $\pi$ and so this property holds automatically. Now, by translating $X$ if necessary, we may assume that $\vec{z}$ is the zero vector. Let the entries of $\vec{x}$ and $\vec{y}$ be given by $\vec{x}=(x_1,\dots,x_d)$ and $\vec{y}=(y_1,\dots,y_d)$. Given that $\vec{z}=\vec{0}$, the condition that the interior angles of $T$ at $\vec{x}$ and $\vec{y}$ are not equal to $\pi/2$ can be expressed algebraically as follows:
\begin{equation}\label{eq:notOrthogx}\langle \vec{x},\vec{x}-\vec{y}\rangle = \sum_{i=1}^dx_i(x_i-y_i) \neq0,\end{equation}
\begin{equation}\label{eq:notOrthogy}\langle \vec{y},\vec{x}-\vec{y}\rangle=\sum_{i=1}^dy_i(x_i-y_i) \neq0,\end{equation}
where $\langle\cdot,\cdot\rangle$ is the standard inner product on $\mathbb{R}^d$. Since $\vec{x}\neq \vec{y}$, by reordering the indices, rescaling $X$ and swapping $\vec{x}$ and $\vec{y}$ if necessary, we can assume that $x_d-y_d=1$ and that $|x_i-y_i|\leq 1$ for all $i\in [d]$. Note that the rescaling may put $\vec{x}$ and $\vec{y}$ in $\mathbb{Q}^d$. For each $i\in [d-1]$, let $x_i-y_i=r_i/q_i$ such that $r_i$ and $q_i>0$ are integers with $q_i$ as small as possible. Note that $q:=\max_{i\in [d-1]}q_i$ is a constant depending on $X$. 

Next, we define
\[t:=\frac{\sum_{i=1}^dy_i(y_i-x_i)}{\sum_{i=1}^d(x_i-y_i)^2}.\]
Note that $t$ is well-defined because $\vec{x}\neq\vec{y}$, and that $t\in\mathbb{Q}$ because $\vec{x},\vec{y}\in\mathbb{Q}^d$. Moreover, $t$ can be regarded as a quantity which depends only on $X$. Also, \eqref{eq:notOrthogy} implies that $t\neq 0$ and \eqref{eq:notOrthogx} implies that $t\neq 1$. So, $X':=\{0,t,1\}$ is a subset of $\mathbb{Q}$ of cardinality 3. Using Lemma~\ref{lem:Behrendd=1}, we let $S'$ be a subset of $[n/4]$ which is $X'$-free and satisfies
\[|S'|\geq (n/4)\cdot e^{-c_{X'}\sqrt{\log(n/4)}}.\]
Note that $X'$ depends on $t$ which, in turn, depends on $X$, and so $c_{X'}$ is a constant that depends on $X$. Define $S\subseteq [n]^d$ by
\[S:=\left\{\vec{s}=(s_1,\dots,s_d)\in [n]^d:\sum_{i=1}^d \left(s_i-\frac{n}{2}\right)(x_i-y_i)\in S'\right\}.\]

To complete the proof, we need to analyze the cardinality of $S$ and show that it is $X$-free. We start with its cardinality. For each $i\in[d-1]$, since $|x_i-y_i|\leq 1$, the number of choices for $s_i\in [n]$ such that
\[-\frac{n}{4(d-1)}\leq \left(s_i-\frac{n}{2}\right)(x_i-y_i)\leq 0\]
is at least $\frac{n}{4(d-1)}$. %So, the number of choices for $(s_1,\dots,s_{d-1})\in [n]^{d-1}$ such that this constraint holds for all $i\in[d-1]$ is $\Omega(n^{d-1})$, where the constant factor depends on $d$ only. 
Since $x_i-y_i$ is a rational with denominator bounded by the constant $q$ for all $i\in [d-1]$, the number of such $s_i$ for which $\left(s_i-\frac{n}{2}\right)(x_i-y_i)$ is also an integer is at least $\left\lfloor \frac{n}{4q(d-1)} \right\rfloor$. Therefore, the number of choices for $(s_1,\dots,s_{d-1})\in [n]^{d-1}$ such that 
\[\sum_{i=1}^{d-1}\left(s_i-\frac{n}{2}\right)(x_i-y_i)\in \mathbb{Z}\cap \left[ -\frac{n}{4}, 0 \right] \]
is $\Omega(n^{d-1})$, where the constant factor depends only on $d$ and $q$ (which, in turn, depends on $X$). Recalling that $x_d-y_d=1$ and that $S'\subseteq [n/4]$, we see that, for any such $(s_1,\dots,s_{d-1})$ and any $a\in S'$, there is a unique choice of $s_d\in [n]$ so that 
\[\sum_{i=1}^d \left(s_i-\frac{n}{2}\right)(x_i-y_i)=a.\]
Thus, we get that
\[|S|=\Omega(n^{d-1}|S'|) =n^d\cdot e^{-c_{X}\sqrt{\log{n}}}\]
for a sufficiently small constant $c_X>0$ depending on $X$, as desired. 

Finally, we show that $S$ is $X$-free. Suppose, to the contrary, that there exists a non-trivial copy of $X$ in $S$. In other words, there exists $\vec{b}\in \mathbb{Q}^d$ and $r\in \mathbb{Q}\cap (0,\infty)$ such that
\[\{\vec{b}, \vec{b}+r \vec{x}, \vec{b}+r\vec{y}\}\subseteq S.\]
Let $\vec{b}=(p_1,\dots,p_d)$ and define
\[a:=\sum_{i=1}^d \left(p_i-\frac{n}{2}\right)(x_i-y_i),\]
\[b:=\sum_{i=1}^d \left(p_i +rx_i-\frac{n}{2}\right)(x_i-y_i),\]
\[c:=\sum_{i=1}^d \left(p_i+ry_i-\frac{n}{2}\right)(x_i-y_i).\]
The condition $\{\vec{b}, \vec{b}+r \vec{x}, \vec{b}+r\vec{y}\}\subseteq S$ is equivalent to $a,b,c\in S'$. We claim that $\{a,b,c\}$ is a non-trivial copy of $\{0,t,1\}$, which will contradict the definition of $S'$ and complete the proof. Since $r\neq0$, we have
\[\frac{a-b}{c-b} = \frac{r\sum_{i=1}^dx_i(y_i-x_i)}{r\sum_{i=1}^d(x_i-y_i)^2} =t.\]
Therefore
\[ b + (c-b) \{0,t,1\} = b + \{0,a-b,c-b\} =  \{a,b,c\}\]
and so $\{a,b,c\}$ is a copy of $\{0,t,1\}$. It is non-trivial because $r\neq 0$ and $\vec{x}\neq\vec{y}$ implies $c-b=r\cdot\sum_{i=1}^d(x_i-y_i)^2$ is non-zero. This contradiction completes the proof. 
\end{proof}

\end{document}